\documentclass[11pt]{article}

\DeclareMathAlphabet{\mathpzc}{OT1}{pzc}{m}{it}

\usepackage{amsmath}
\usepackage{amssymb}
\usepackage{amsfonts}
\usepackage{amsthm}
\usepackage{mathrsfs}
\usepackage{ifsym}
\usepackage{fontenc}
\usepackage{textcomp}
\usepackage{tipa}
\usepackage{enumerate}
\usepackage[pdftex]{color, graphicx}

\begin{document}
 
\title{{\bf New proofs of two identities of Ramanujan}}         
\author{Patrick Morton}        
\date{Jan. 10, 2024}          
\maketitle

\begin{abstract}  A proof of two identities of Ramanujan involving theta functions of level $7$ is given which uses a specific modular function for $\Gamma_1(7)$ and Klein's projective representation of $PSL(2,7)$ into $PSL(3, \mathbb{C})$.  Four identities of Berndt and Zhang are derived as algebraic corollaries of the main proof.
\end{abstract}

\section{Introduction}

Define the following functions, using Ramanujan's notation
$$f(a,b) = \sum_{n=-\infty}^\infty{a^{n(n+1)/2} b^{n(n-1)/2}}, \ \ |ab| < 1,$$
where
$$f(a,b) = (-a;ab)_\infty(-b;ab)_\infty (ab;ab)_\infty,$$
with
$$(a;q)_\infty = \prod_{n=0}^\infty{(1-aq^n)}, \ \ |q| < 1.$$
We set
\begin{equation}
u = q^{1/56} f(-q^3,-q^4), \ \ v = q^{9/56} f(-q^2,-q^5), \ \ w = q^{25/56} f(-q,-q^6).
\label{eqn:1}
\end{equation}
 Also, let $\eta(\tau) = q^{1/24}\prod_{n=1}^\infty{(1-q^n)}$ denote the Dedekind $\eta$-function and let
 \begin{equation}
j_7^*(\tau) = \left(\frac{\eta(\tau)}{\eta(7\tau)}\right)^4+13 + 49 \left(\frac{\eta(7\tau)}{\eta(\tau)}\right)^4.
 \label{eqn:2}
 \end{equation}
 In \cite{bz}, Berndt and Zhang give proofs of the identities in the following theorem, which they say are ``dissimilar to all other theta-function identities existing in the literature" (written in 1994).

\newtheorem{thm}{Theorem}
\begin{thm}
With $u,v, w$ defined as in (\ref{eqn:1}), we have 
\begin{align}
\label{eqn:3} \frac{u^2}{v}-\frac{v^2}{w}+\frac{w^2}{u} &= 0,\\
\label{eqn:4} \frac{v}{u^2}-\frac{w}{v^2}+\frac{u}{w^2} &= \frac{\eta(\tau)}{\eta^2(7\tau)} j_7^*(\tau)^{1/3}.
\end{align}
\label{thm:1}
\end{thm}

In \cite{bz}, the identity (\ref{eqn:4}) is expressed using the notation $f(-q) = f(-q,-q^2)$ in the equivalent form
$$\frac{v}{u^2}-\frac{w}{v^2}+\frac{u}{w^2} = \frac{f(-q)}{q^{7/8}f^2(-q^7)}\left(\frac{f^4(-q)}{f^4(-q^7)}+13q +49q^2\frac{f^4(-q^7)}{f^4(-q)}\right)^{1/3}.$$
See also Entry 32 in \cite[p. 176]{ber}.  These identities are from the unorganized material in Ramanujan's notebooks.  Identity (\ref{eqn:3}) is related to Klein's equation; see Corollary \ref{cor:1} after Theorem \ref{thm:3}.\medskip

The proof in \cite[Thm. 3.1]{bz} uses a general construction of a modular function for the congruence group
$$\Gamma^0(p) = \Big\{\left(\begin{array}{cc}a & b \\c & d\end{array}\right) \in \Gamma(1): b \equiv 0 \ (\textrm{mod} \ p)\Big\},$$
where $p$ is a prime satisfying $p \equiv \pm 1, \pm 7$ (mod $16$).  The identities in Theorem \ref{thm:1}, as well as four other similar identities, given in Examples 1-3 in \cite[pp. 238-242]{bz}, are all derived with $p = 7$, using the same construction with different parameters and producing different modular functions.  The identities follow from investigating the behavior of these modular functions at the two cusps $0, \infty$ of $\Gamma^0(7)$.  In addition, the proofs in \cite{bz} require some complicated relationships involving the sine function.  In the introduction, the authors mention that they would like to find a proof which does not use the theory of modular forms.  Such a proof has been provided by Cooper in \cite[Thm. 7.14, p. 440]{cp}, which makes use of an elliptic function identity of Weierstrass \cite[pp. 65-66, 76]{cp}, along with Lambert series, several forms of an identity of Halphen \cite[pp. 67, 71-72]{cp}, and some complicated relationships involving the functions $\alpha = v/w, \beta = -u/v, \gamma = w/u$.  The four additional identities in Examples 2 and 3 of \cite{bz} are not mentioned in \cite[Ch. 7]{cp}.  \medskip

In this note I will give unified proofs of the identities (\ref{eqn:3}) and (\ref{eqn:4}) and the four identities of Berndt and Zhang by making use of the properties of the single modular function
\begin{equation}
h(\tau) = \frac{u v^2}{w^3} = q^{-1} \prod_{k=0}^\infty{\frac{(1-q^{7k+4})(1-q^{7k+3})(1-q^{7k+5})^2(1-q^{7n+2})^2}{(1-q^{7n+6})^3 (1-q^{7k+1})^3}}.
\label{eqn:5}
\end{equation}
This product formula follows from the product expansions
\begin{align*}
u & = q^{1/56} \prod_{k=0}^\infty{(1-q^{7k+3})(1-q^{7k+4})(1-q^{7k+7})},\\
v & = q^{9/56} \prod_{k=0}^\infty{(1-q^{7k+2})(1-q^{7k+5})(1-q^{7k+7})},\\
w & = q^{25/56} \prod_{k=0}^\infty{(1-q^{7k+1})(1-q^{7k+6})(1-q^{7k+7})}.
\end{align*} 
The properties needed for the proof can be enumerated as follows.  These are all known facts, and can be found in \cite{du}, \cite{elk} and \cite{m}.
\begin{enumerate}[1.]
\item $h(\tau)$ is a modular function for the congruence group
$$\Gamma_1(7) = \Big\{\left(\begin{array}{cc}a & b \\c & d\end{array}\right) \in \Gamma(1): a \equiv d \equiv 1, c \equiv 0 \ (\textrm{mod} \ 7)\Big\},$$
(see \cite{du}, \cite{ds}), and is a Hauptmodul for this group, i.e. it generates the field $\textsf{K}_{\Gamma_1(7)}$ of modular functions for $\Gamma_1(7)$ which are defined over the field $k = \mathbb{Q}(\zeta_7)$, where $\zeta_7 = e^{2\pi i/7}$, so that $\textsf{K}_{\Gamma_1(7)} = k(h(\tau))$.
\item $\displaystyle h\left(\frac{2\tau-1}{7\tau-3}\right) = \frac{h(\tau)-1}{h(\tau)} = \frac{u^2w}{v^3}$, for $\textrm{Im}(\tau) > 0$.
\item $h(-1/7\tau) = \frac{h-r}{(1-r)h-1}$, where $r = -(\zeta+\zeta^6)^2 (\zeta^3+\zeta^4)^{-1}$ (with $\zeta = e^{2 \pi i/7}$) and $\frac{1}{1-r} = -(\zeta^3+\zeta^4)^2(\zeta^2+\zeta^5)^{-1}$ are roots of $x^3-8x^2+5x+1$.  In particular, $\lim_{\tau \rightarrow \infty i} h(\frac{-1}{7\tau}) = \frac{1}{1-r}$.
\item From 2 and 3 and some additional considerations it follows that $\displaystyle \frac{\eta^4(\tau)}{\eta^4(7\tau)} = \frac{h^3-8h^2+5h+1}{h(h-1)}$, with $h = h(\tau)$ and $\textrm{Im}(\tau) > 0$.
\end{enumerate}
\medskip

Here I will assume 1 and prove 2, 3 and 4.  The relation (\ref{eqn:3}) is a corollary of the proof of fact 2, and the identity (\ref{eqn:4}) follows from facts 2 and 4, using (\ref{eqn:3}).   I will show that these facts all follow from Klein's projective representation of the group $SL_2(\mathbb{Z})$ into $PSL(3, \mathbb{C})$.  See Klein's paper \cite[pp. 106-107]{kl} and \cite{elk}.  The main tool used here is the transformation formulas for theta functions with characteristics, which is applied in the proof of Theorem \ref{thm:2}.  The proofs of properties 3 and 4 use the fact that the fields of modular functions considered here, namely, those whose corresponding congruence groups are $\Gamma_1(7), \Gamma_0(7)$ and the full modular group $\Gamma = SL_2(\mathbb{Z})$, are algebraic function fields of genus $0$; and that the cusps of these groups can be viewed as prime divisors at infinity of the corresponding algebraic function fields.  Normally, the choice of the prime divisor ``at infinity" in the rational function field $k(t)$ depends on the choice of a generator $t$.  In the context of the theory of modular functions the natural choice is $t = j(\tau)$. \medskip

An advantage of the proofs given here is that the four additional identities proved by Berndt and Zhang have purely algebraic derivations, using Ramanujan's two identities and the relationships holding between the function $h$ and $u,v,w$.  Hence, the only modular function which is necessary for these proofs is $h(\tau)$, although the auxiliary functions $s(\tau)=u/w$ and $t(\tau)=v/w$ are used in Theorems \ref{thm:2} and \ref{thm:3} to prove the necessary properties of $h(\tau)$.  The identities we prove in Theorems \ref{thm:6} and \ref{thm:8} are actually simpler than the identities proved by Berndt and Zhang in their Examples 2 and 3, since the latter identities are the cubes of the identities proved here.  In addition, the sine function does not appear explicitly in these proofs.  Instead, we use straightforward calculations in the field of $7$-th roots of unity.  Overall, the proofs given here have a largely algebraic character, while the proofs in \cite{bz} are more analytic.

\medskip

\section{Klein's projective representation}

We start with the following functions from \cite[p. 157]{du}:
\begin{align}
\label{eqn:6} s(\tau) = & \ =  \frac{u}{w} = q^{-3/7} \prod_{n \ge 1}{\frac{(1-q^{7n-3})(1-q^{7n-4})}{(1-q^{7n-1})(1-q^{7n-6})}} = \textstyle e\left(\frac{1}{7}\right) \frac{\theta {1/7 \atopwithdelims [] 1}(7\tau)}{\theta {5/7 \atopwithdelims [] 1}(7\tau)},\\
\label{eqn:7} t(\tau) = & \ = \frac{v}{w} = q^{-2/7} \prod_{n \ge 1}{\frac{(1-q^{7n-2})(1-q^{7n-5})}{(1-q^{7n-1})(1-q^{7n-6})}} = \textstyle e\left(\frac{1}{14}\right) \frac{\theta {3/7 \atopwithdelims [] 1}(7\tau)}{\theta {5/7 \atopwithdelims [] 1}(7\tau)}.
\end{align}
In this and the following formulas, $e(z) = e^{2\pi i z}$.  See \cite{fk} and \cite{du} for all the facts we need about theta functions with characteristics.  From the infinite products in (\ref{eqn:6}) and (\ref{eqn:7}) we get 
\begin{align}
\label{eqn:8} s(\tau) t^2(\tau) = & \ q^{-1} \prod_{n \ge 1}{\frac{(1-q^{7n-3})(1-q^{7n-4})(1-q^{7n-2})^2(1-q^{7n-5})^2}{(1-q^{7n-1})^3(1-q^{7n-6})^3}}\\
\label{eqn:9} = & \ h(\tau).
\end{align}

\begin{thm} The following identities hold.
\begin{align}
\label{eqn:13} s\left(\frac{-1}{\tau}\right) & = \frac{\varepsilon_1 s(\tau) - \eta_3 t(\tau) + 1}{s(\tau)-\varepsilon_1 t(\tau) -\eta_3}\\
\label{eqn:14} & = \frac{\varepsilon_1 s(\tau) +(\varepsilon_1^2-\varepsilon_1-1) t(\tau) + 1}{s(\tau)-\varepsilon_1 t(\tau) +\varepsilon_1^2-\varepsilon_1-1},
\end{align}
\begin{align}
\label{eqn:15} t\left(\frac{-1}{\tau}\right) & = \frac{-\eta_3 s(\tau) - t(\tau) - \varepsilon_1}{s(\tau)-\varepsilon_1 t(\tau) -\eta_3}\\
\label{eqn:16} & = \frac{(\varepsilon_1^2-\varepsilon_1-1) s(\tau) - t(\tau) - \varepsilon_1}{s(\tau)-\varepsilon_1 t(\tau) + \varepsilon_1^2-\varepsilon_1-1}.
\end{align}
Here, $\varepsilon_1 = 1 + \eta_1 = 1 + \zeta_7 +\zeta_7^6$, and $\eta_i = \zeta_7^i + \zeta_7^{-i}$, with $\zeta_7 = e^{2\pi i/7}$.
\label{thm:2}
\end{thm}

\begin{proof} 
From (\ref{eqn:6}) and the transformation formula
$$\theta {\epsilon \atopwithdelims [] \epsilon'}\left(\frac{-1}{\tau}\right) = \textstyle e(-1/8) \sqrt{\tau} e\left(\frac{\epsilon \epsilon'}{4}\right) \theta {\epsilon' \atopwithdelims [] -\epsilon}(\tau)$$
we have
\begin{align}
\notag s(-1/\tau) &= \textstyle e(\frac{1}{7}) \displaystyle \frac{\theta {1/7 \atopwithdelims [] 1}(-7/\tau)}{\theta {5/7 \atopwithdelims [] 1}(-7/\tau)}\\
\notag & = \frac{\theta {1 \atopwithdelims [] -1/7}(\tau/7)}{\theta {1 \atopwithdelims [] -5/7}(\tau/7)} = \frac{\theta {1 \atopwithdelims [] 1/7}(\tau/7)}{\theta {1 \atopwithdelims [] 5/7}(\tau/7)}\\
\label{eqn:17} & = \frac{\sum_{k=0}^6{\theta {\frac{1+2k}{7} \atopwithdelims [] 1}(7\tau)}}{\sum_{k=0}^6{\theta {\frac{1+2k}{7} \atopwithdelims [] 5}(7\tau)}},
\end{align}
where the last equality follows from \cite[(4.4)]{du}:
$$\theta {\epsilon \atopwithdelims [] \epsilon'}(\tau) = \sum_{k=0}^{N-1}{\theta {\frac{\epsilon+2k}{N} \atopwithdelims [] N\epsilon'}(N^2\tau)}.$$
Using $\theta {\epsilon \atopwithdelims [] \epsilon'}(\tau) = e(\mp \frac{\epsilon m}{2}) \theta {\pm \epsilon + 2\ell \atopwithdelims [] \pm \epsilon'+2m}(\tau)$, we have $\theta {1 \atopwithdelims [] 1}(\tau) = 0$, and the numerator in (\ref{eqn:17}) becomes
\begin{align*}
A &= \textstyle \theta {1/7 \atopwithdelims [] 1}(7\tau) +  \theta {3/7 \atopwithdelims [] 1}(7\tau) +  \theta {5/7 \atopwithdelims [] 1}(7\tau)\\
& \ \ \textstyle + e\left({\frac{9}{14}}\right)  \theta {5/7 \atopwithdelims [] 1}(7\tau) + e\left(\frac{11}{14}\right) \theta {3/7 \atopwithdelims [] 1}(7\tau) + e\left(\frac{13}{14}\right)  \theta {1/7 \atopwithdelims [] 1}(7\tau)\\
& = \textstyle (1+ e\left(\frac{13}{14}\right)) \theta {1/7 \atopwithdelims [] 1}(7\tau) + (1+e\left(\frac{11}{14}\right)) \theta {3/7 \atopwithdelims [] 1}(7\tau)\\
& \ \ \textstyle +(1+ e\left(\frac{9}{14}\right)) \theta {5/7 \atopwithdelims [] 1}(7\tau).
\end{align*}
Hence
\begin{align*}
\frac{A}{\theta {5/7 \atopwithdelims [] 1}(7\tau)} & = \textstyle (1+ e\left(\frac{13}{14}\right)) e(\frac{-1}{7}) s(\tau) + (1+e\left(\frac{11}{14}\right)) e(\frac{-1}{14}) t(\tau)\\
& \ \ \textstyle +(1+ e\left(\frac{9}{14}\right))\\
& = \textstyle (e(\frac{6}{7})-e(\frac{2}{7})) s(\tau) + (e(\frac{-1}{14})+e(\frac{5}{7})) t(\tau) + 1+ e(\frac{9}{14})\\
& = (\zeta^{12}-\zeta^4)s(\tau) + (\zeta^{13}+\zeta^{10}) t(\tau) + (1+\zeta^9),
\end{align*}
where $\zeta = e^{2\pi i/14}$.  Similarly, the denominator $B$ of (\ref{eqn:17}) satisfies
\begin{align*}
B & = \textstyle \theta {1/7 \atopwithdelims [] 5}(7\tau) +  \theta {3/7 \atopwithdelims [] 5}(7\tau) +  \theta {5/7 \atopwithdelims [] 5}(7\tau)\\
& \ \  \textstyle + \theta {9/7 \atopwithdelims [] 5}(7\tau) +  \theta {11/7 \atopwithdelims [] 5}(7\tau) +  \theta {13/7 \atopwithdelims [] 5}(7\tau)\\
& = \textstyle e\left(\frac{1}{7}\right) \theta {1/7 \atopwithdelims [] 1}(7\tau) + e\left(\frac{3}{7}\right) \theta {3/7 \atopwithdelims [] 1}(7\tau) +  e\left(\frac{5}{7}\right) \theta {5/7 \atopwithdelims [] 1}(7\tau)\\
& \ \  \textstyle + e\left(\frac{2}{7}\right) \theta {9/7 \atopwithdelims [] 1}(7\tau) +  e\left(\frac{4}{7}\right) \theta {11/7 \atopwithdelims [] 1}(7\tau) +  e\left(\frac{6}{7}\right) \theta {13/7 \atopwithdelims [] 1}(7\tau)\\
& = \textstyle e\left(\frac{1}{7}\right) \theta {1/7 \atopwithdelims [] 1}(7\tau) + e\left(\frac{3}{7}\right) \theta {3/7 \atopwithdelims [] 1}(7\tau) +  e\left(\frac{5}{7}\right) \theta {5/7 \atopwithdelims [] 1}(7\tau)\\
& \ \  \textstyle + e\left(\frac{13}{14}\right) \theta {5/7 \atopwithdelims [] 1}(7\tau) +  e\left(\frac{5}{14}\right) \theta {3/7 \atopwithdelims [] 1}(7\tau) +  e\left(\frac{11}{14}\right) \theta {1/7 \atopwithdelims [] 1}(7\tau),
\end{align*}
which leads to
\begin{align*}
\frac{B}{\theta {5/7 \atopwithdelims [] 1}(7\tau)} & = \textstyle (e\left(\frac{1}{7}\right)+ e\left(\frac{11}{14}\right)) e(\frac{-1}{7}) s(\tau) + (e\left(\frac{3}{7}\right)+e\left(\frac{5}{14}\right)) e(\frac{-1}{14}) t(\tau)\\
& \ \ \textstyle +(e\left(\frac{5}{7}\right)+ e\left(\frac{13}{14}\right))\\
& = \textstyle (1+e(\frac{9}{14})) s(\tau) + (e(\frac{5}{14})+e(\frac{4}{14})) t(\tau) + (e(\frac{5}{7})+ e(\frac{13}{14}))\\
& = (1+\zeta^9)s(\tau) + (\zeta^{5}+\zeta^{4}) t(\tau) + (\zeta^{10}-\zeta^6).
\end{align*}
This gives that
\begin{align*}
s(-1/\tau) &= \frac{A}{B} = \frac{(\zeta^{12}-\zeta^4)s(\tau) + (\zeta^{13}+\zeta^{10}) t(\tau) + (1+\zeta^9)}{(1+\zeta^9)s(\tau) + (\zeta^{5}+\zeta^{4}) t(\tau) + (\zeta^{10}-\zeta^6)}\\
& = \frac{(\zeta_7^6-\zeta_7^2)s(\tau) + (\zeta_7^5-\zeta_7^3) t(\tau) + (1-\zeta_7)}{(1-\zeta_7)s(\tau) + (\zeta_7^2-\zeta_7^6) t(\tau) + (\zeta_7^5-\zeta_7^3)}\\
& = \frac{-\zeta_7^2(1+\zeta_7)(1+\zeta_7^2)s(\tau) -\zeta_7^3 (1+\zeta_7) t(\tau) + 1}{s(\tau) + \zeta_7^2(1+\zeta_7)(1+\zeta_7^2) t(\tau) -\zeta_7^3(1+\zeta_7)},
\end{align*}
with $\zeta_7 = \zeta^2 = e^{2\pi i/7}$.  Since $-\zeta_7^2(1+\zeta_7)(1+\zeta_7^2) = 1+\zeta_7+\zeta_7^6 = \varepsilon_1$ and $\zeta_7^3+\zeta_7^4 = \eta_3$, this yields the formulas (\ref{eqn:13}) and (\ref{eqn:14}). \medskip

A similar computation gives (\ref{eqn:15}) and (\ref{eqn:16}), or we may argue as follows.  Since the denominator of $t(-1/\tau)$ coincides with the denominator of $s(-1/\tau)$, we have that
$$t\left(\frac{-1}{\tau}\right) = \frac{a s(\tau)+b t(\tau) + c}{s(\tau)-\varepsilon_1 t(\tau) +\varepsilon_1^2-\varepsilon_1-1},$$
for some $a,b,c$.  Since $\tau \rightarrow -1/\tau$ has order $2$, the $3 \times 3$ matrix
\begin{equation*}
\textsf{A} = 
\left(
\begin{array}{ccc}
\varepsilon_1  & \varepsilon_1^2-\varepsilon_1-1  & 1  \\
a  & b  &  c \\
1  &  -\varepsilon_1  &  \varepsilon_1^2-\varepsilon_1-1 
\end{array}
\right),
\end{equation*}
which satisfies the projective equation
\begin{equation}
\textsf{A} \left(
\begin{array}{c}
s(\tau) \\
t(\tau) \\
1 
\end{array}
\right) = \left(
\begin{array}{c}
s(-1/\tau) \\
t(-1/\tau) \\
1 
\end{array}
\right),
\label{eqn:31}
\end{equation}
or, equivalently,
\begin{equation*}
\textsf{A} \left(
\begin{array}{c}
u(\tau) \\
v(\tau) \\
w(\tau) 
\end{array}
\right) = \left(
\begin{array}{c}
u(-1/\tau) \\
v(-1/\tau) \\
w(-1/\tau) 
\end{array}
\right),
\end{equation*}
must also satisfy $\textsf{A}^2 = dI$, and checking off-diagonal elements in the matrix $\textsf{A}^2$ shows that this can hold if and only if $a = \varepsilon_1^2-\varepsilon_1-1, b=-1, c= -\varepsilon_1$.  Hence,
\begin{equation}
\textsf{A} = 
\left(
\begin{array}{ccc}
\varepsilon_1  & \varepsilon_1^2-\varepsilon_1-1  & 1  \\
\varepsilon_1^2-\varepsilon_1-1  & -1  &  -\varepsilon_1 \\
1  &  -\varepsilon_1  &  \varepsilon_1^2-\varepsilon_1-1 
\end{array}
\right),
\label{eqn:18}
\end{equation}
proving (\ref{eqn:15}) and (\ref{eqn:16}).  Note that $\textsf{A}$ is a real symmetric matrix.
\end{proof}

\begin{thm} We have the formula
$$h\left(\frac{2\tau-1}{7\tau-3}\right)= \frac{h(\tau)-1}{h(\tau)}, \ \ \tau \in \mathbb{H}.$$
\label{thm:3}
\end{thm}

\begin{proof}
First note the following: if
$$A(\tau) = \frac{2\tau-1}{3\tau-7}, \ \ A^{-1}(\tau) = \frac{3\tau-1}{7\tau-2},$$
and
$$\alpha(\tau) = \frac{-1}{\tau}, \ \ \beta(\tau) = \tau+1,$$
then
$$A^{-1}(\tau) = \beta \circ \alpha \circ \beta^2 \circ \alpha \circ \beta^4 \circ \alpha(\tau).$$
Using further that
\begin{equation}
\textsf{B} \left(
\begin{array}{c}
s(\tau) \\
t(\tau) \\
1 
\end{array}
\right) = \left(
\begin{array}{c}
s(\tau+1) \\
t(\tau+1) \\
1 
\end{array}
\right),
\label{eqn:32}
\end{equation}
where
\begin{equation*}
\textsf{B} = 
\left(
\begin{array}{ccc}
\zeta_7^4  & 0  &  0 \\
 0 & \zeta_7^5  & 0  \\
 0 & 0  &  1
\end{array}
\right),
\end{equation*}
equations (\ref{eqn:31}) and (\ref{eqn:32}) imply that
\begin{equation}
\textsf{T} \left(
\begin{array}{c}
s(\tau) \\
t(\tau) \\
1 
\end{array}
\right) = \left(
\begin{array}{c}
s(A^{-1}(\tau)) \\
t(A^{-1}(\tau)) \\
1 
\end{array}
\right),
\label{eqn:29}
\end{equation}
where (this time with $\zeta = \zeta_7$ and $\eta_i$ as in Theorem \ref{thm:2})
\begin{equation}
\textsf{T} = \textsf{B} \textsf{A} \textsf{B}^2 \textsf{A} \textsf{B}^4 \textsf{A} = 
\eta_1\eta_3^2 \left(
\begin{array}{ccc}
0  &  -\zeta^4 & 0  \\
0  & 0  &  \zeta  \\
\zeta^2 & 0 &  0 
\end{array}
\right).
\label{eqn:30}
\end{equation}

Since the matrix corresponding to $A(\tau) = \frac{2\tau-1}{7\tau-3}$ lies in $\Gamma_0(7)$, and $\Gamma_1(7)$ is normal in the former group, we know that $h(A(\tau))$ is also a Hauptmodul for $\textsf{K}_{\Gamma_1(7)}$.  Hence,
$$h(A(\tau)) = \frac{ah+b}{ch+d}, \ \ h = h(\tau) = st^2.$$
Squaring the matrix $\textsf{T} = \textsf{B} \textsf{A} \textsf{B}^2 \textsf{A} \textsf{B}^4 \textsf{A}$ gives
\begin{equation*}
\textsf{T}^2 = \eta_1^2 \eta_3^4\left(
\begin{array}{ccc}
0  & 0  & -\zeta^5  \\
 \zeta^3 & 0  & 0  \\
0  & -\zeta^6  &  0 
\end{array}
\right)
\end{equation*}
and applying $\textsf{T}^2$ to the vector $(s(\tau), t(\tau), 1)^{tr}$ yields
\begin{equation}
h(A(\tau)) = h(A^{-2}(\tau)) = \frac{-\zeta^5\zeta^6}{(-\zeta^6)^3}\frac{s^2}{t^3} = \frac{s^2}{t^3}.
\label{eqn:33a}
\end{equation}
Hence, we have
$$\frac{ast^2+b}{cst^2+d} = \frac{s^2}{t^3} \ \Longrightarrow \ ast^5+bt^3 = cs^3t^2+ds^2.$$
Comparing the leading term of the $q$-expansions of both sides of the last equation yields
$$aq^{-13/7}+bq^{-6/7} + \cdots = cq^{-13/7} + dq^{-6/7} + \cdots,$$
which implies that $a = c$.  Hence, we can assume $a = c = 1$.  \medskip

From (\ref{eqn:29}) and (\ref{eqn:30}) we find
\begin{align*}
h(A^2(\tau)) & = h(A^{-1}(\tau)) = s(A^{-1}(\tau))t^2(A^{-1}(\tau))\\
& = \frac{(-\zeta^4)(\zeta)^2 t}{\zeta^6 s^3}\\
& = -\frac{t}{s^3}.
\end{align*}
Thus,
$$\frac{(b+1)st^2+b(d+1)}{(d+1)st^2+b+d^2} = -\frac{t}{s^3},$$
which implies that
$$(b+1)s^4t^2 + b(d+1)s^3 = -(d+1)st^3-(b+d^2)t.$$
Comparing leading terms of the $q$-expansions gives that
$$(b+1)q^{-16/7} + b(d+1) q^{-9/7} + \cdots = -(d+1)q^{-9/7} + (b+d^2)q^{-2/7} + \cdots$$
and implies that $b=-1$.  Now the condition that the cube of the matrix
\begin{equation*}
S = \left(
\begin{array}{cc}
1 & -1  \\
1  &  d  
\end{array}
\right)
\end{equation*}
must be a diagonal matrix yields the condition $d^2+d=0$, hence $d = 0, -1$. But $d = -1$ is impossible, since $S$ is nonsingular, so that $d=0$.  Thus,
$h(A(\tau)) = \frac{h-1}{h}$, which is the assertion of Theorem \ref{thm:3}.
\end{proof}

\newtheorem{cor}{Corollary}

\begin{cor} With $s = s(\tau), t = t(\tau)$, we have
$$s^3 -st^3+t = 0,$$
which is equivalent to
$$u^3 w - v^3 u +w^3 v = 0.$$
\label{cor:1}
\end{cor}
\begin{proof}
This is immediate from the relation $h(A(\tau)) = \frac{s^2}{t^3} = \frac{st^2-1}{st^2}$.
\end{proof}

Note that Corollary \ref{cor:1} is equivalent to Ramanujan's identity (\ref{eqn:3}).  Furthermore, this corollary shows that $(X,Y,Z) = (s(\tau), -1, t(\tau))$ satisfies Klein's equation
$$X^3Y+Y^3Z+Z^3X = 0.$$
See \cite[pp. 84-85]{elk}, where the functions $\textsf{x},\textsf{y},\textsf{z}$ satisfying Klein's equation are given by
$$\textsf{x} = -\eta(\tau)^3 w, \ \ \textsf{y} = \eta(\tau)^3 v, \ \ \textsf{z} = \eta(\tau)^3 u.$$
Theorem \ref{thm:3} implies that 
\begin{equation}
\lim_{\tau \rightarrow \infty i} h\left(\frac{2\tau-1}{7\tau-3}\right) = h(2/7) = 1,
\label{eqn:28}
\end{equation}
which we will use to prove Theorems \ref{thm:4} and \ref{thm:5} below. \medskip

Now we draw the following further consequence from Theorem \ref{thm:2}.  As $\tau \rightarrow \infty i$, Theorem \ref{thm:2} gives that
\begin{align}
\label{eqn:19} \lim_{\tau \rightarrow \infty i} h\left(\frac{-1}{7\tau}\right) & = \lim_{\tau \rightarrow \infty i} s\left(\frac{-1}{7\tau}\right) t^2\left(\frac{-1}{7\tau}\right)\\
\notag & = \textrm{lim}_{q \rightarrow 0} \frac{\varepsilon_1 - \eta_3 \frac{t(7\tau)}{s(7\tau)} + \frac{1}{s(7\tau)}}{1-\varepsilon_1 \frac{t(7\tau)}{s(7\tau)} -\frac{\eta_3}{s(7\tau)}} \left(\frac{-\eta_3 - \frac{t(7\tau)}{s(7\tau)} -  \frac{\varepsilon_1}{s(7\tau)}}{1-\varepsilon_1 \frac{t(7\tau)}{s(7\tau)} -\frac{\eta_3}{s(7\tau)}}\right)^2\\
\label{eqn:20} & = \varepsilon_1 \eta_3^2 = -(\zeta^3+\zeta^4)^2(\zeta^2+\zeta^5)^{-1} = \frac{1}{1-r},
\end{align}
where $r = -(\zeta+\zeta^6)^2 (\zeta^3+\zeta^4)^{-1}$ (with $\zeta = e^{2 \pi i/7}$), as in fact 3 of the Introduction. \medskip

By the fact 1 listed in the Introduction and $\textsf{K}_{\Gamma(1)} \subset \textsf{K}_{\Gamma_1(7)}$, the $j$-function is a rational function of $h(\tau)$, so we have
$$j(\tau) = \frac{P(h(\tau))}{Q(h(\tau))}, \ \ (P(X),Q(X)) = 1.$$
Thus, $h=h(\tau)$ satisfies the equation
$$P(h(\tau)) - j(\tau) Q(h(\tau)) = 0.$$
By the Extended $q$-Expansion Principle \cite[p. 62]{sch} and the fact that the coefficients in the $q$-expansion of $h(\tau)$ lie in $\mathbb{Q}$, we have $P(X), Q(X) \in \mathbb{Q}[X]$.  Let $\wp_\infty$ be the prime divisor of the algebraic function field $k(j)$ dividing the denominator of $j = j(\tau)$.  There is a 1-1 correspondence between prime divisors of the field $k(j)$ and $\mathbb{H}/\Gamma(1) \cup \{\infty\}$, and $\wp_\infty$ corresponds to the cusp $\infty$.  The prime divisor $\wp_\infty$ of $k(j)$ extends to a number of prime divisors of $k(h)$ ``at infinity'', and these prime divisors correspond 1-1 to the cusps of $\mathbb{H}/\Gamma_1(7)$.  It is known that there are $6$ such cusps (see \cite[p. 102]{ds}).  One of these cusps is $\tau = \infty i$, at which $h(\tau)$ takes the value $\infty$, by (\ref{eqn:5}); this cusp corresponds therefore to the pole divisor $\mathfrak{p}_\infty$ of $h$.  Since $\mathfrak{p}_\infty \mid \wp_\infty$, it follows that $\textrm{deg}(P(X)) > \textrm{deg}(Q(X))$.

\newtheorem{lem}{Lemma}

\begin{lem} The values of $h(\tau)$ at the six cusps of $\Gamma_1(7)$ are $\{0, 1, \infty, r, r' = \frac{1}{1-r}, r'' = \frac{r-1}{r}\}$, where $r, r', r''$ are the roots of $x^3-8x^2+5x+1$ in $k = \mathbb{Q}(\zeta_7)$.  Hence, the prime divisors of $\wp_\infty$ in $k(h)$ are the primes
$$\mathfrak{p}_\infty,  \mathfrak{p}_0,  \mathfrak{p}_1,  \mathfrak{p}_r,  \mathfrak{p}_{r'},  \mathfrak{p}_{r''},$$
where $\mathfrak{p}_a$ is defined by $h = h(\tau) \equiv a$ (mod $\mathfrak{p}_a$) in $\textsf{K}_{\Gamma_1(7)}$ and the principal divisor of $h$ is $(h) = \frac{\mathfrak{p}_0}{\mathfrak{p}_\infty}$.
\label{lem:1}
\end{lem}

\begin{proof}
We make use of (\ref{eqn:28}), according to which $h(2/7) = 1$.  Hence, one of the values of $h$ at a cusp is $1$.  Further, the infinite product for $h(\tau)$ implies that $h(\tau) \neq 0$ for $\tau \in \mathbb{H}$.  But the numerator divisor $\mathfrak{p}_0$ of $h$ satisfies $h \equiv 0$ (mod $\mathfrak{p}_0$), so that $0$ is also an $h$-value at a cusp.  We appeal now to (\ref{eqn:19})-(\ref{eqn:20}), which imply that $h(0) = \frac{1}{1-r}$, where $r$ and $r' = \frac{1}{1-r}$ are roots of the polynomial $x^3-8x^2+5x+1$.  It follows that the prime divisor $\mathfrak{p}_{r'}$ satisfying $h \equiv r'$ (mod $\mathfrak{p}_{r'}$) divides $\wp_\infty$.  The relation between $j$ and $h$ implies that $Q(h) \equiv Q(r') = 0$; otherwise, $j(\tau)$ would have a finite value (mod $\mathfrak{p}_{r'}$).  Since $Q(X)$ has coefficients in $\mathbb{Q}$, the irreducible polynomial $X^3-8X^2+5X+1$ divides $Q(X)$.  This shows that all three roots of $Q(X)$ are $h$-values at cusps.
\end{proof}

\begin{thm}
We have the identity
\begin{equation}
h\left(\frac{-1}{7\tau}\right) = \frac{h-r}{(1-r)h-1}, \ \ h = h(\tau),
\label{eqn:21}
\end{equation}
where $r = -(\zeta+\zeta^6)^2 (\zeta^3+\zeta^4)^{-1} = 3(\zeta^5+\zeta^2) + \zeta^3 + \zeta^4 + 4$ and $\zeta = e^{2\pi i/7}$.
\label{thm:4}
\end{thm}

\noindent {\bf Remark.} This identity is not necessary for the proof of (\ref{eqn:3}) and (\ref{eqn:4}), but its proof is convenient and short and helps to specify the cusps at which $h(\tau)$ takes the values in Lemma \ref{lem:1}.

\begin{proof}
The function $h(-1/7\tau)$ is a modular function for $\Gamma_1(7)$, since the mapping $\tau \rightarrow \frac{-1}{7\tau}$ normalizes $\Gamma_1(7)$.  It follows that $h(-1/7\tau)$ is a Hauptmodul for $\Gamma_1(7)$, which implies that
$$h\left(\frac{-1}{7\tau}\right) = \frac{ah+b}{ch+d}, \ \ h = h(\tau), \ \ ad-bc \neq 0.$$
As an automorphism of $k(h)$, the mapping $\tau \rightarrow \frac{-1}{7\tau}, h \rightarrow \frac{ah+b}{ch+d}$ must permute the prime divisors at infinity.  Thus, by (\ref{eqn:19})-(\ref{eqn:20}), the value $h(\infty i) = \infty$ is mapped to the value $\frac{1}{1-r}$, so we may assume that
$$h\left(\frac{-1}{7\tau}\right) = \frac{h+b}{(1-r)h-1},$$
where $d= -1$ for the mapping to have order $2$.  The value $0$ cannot map to $1$, since this would imply $b=-1$; but then the value $r$ would map to
$$\frac{r-1}{(1-r)r-1} = \textstyle \frac{5}{49}r^2-\frac{43}{49}r+\frac{41}{49},$$
which is not a conjugate of $r$.  Hence, the value $0$ maps to $-b$ and $1$ maps to $(b + 1)(r^2 - 8r + 5)$, both of which must be roots of $f(x) = x^3-8x^2+5x+1$.  Since the sum of the roots of $f(x)$ is $8$, this implies that
$$\frac{1}{1-r} -b + (b + 1)(r^2 - 8r + 5) = 8,$$
and solving for $b$ gives that $b = \frac{r+1}{r^2-8r+4} = -r$.
\end{proof}

\begin{cor} The function $h(\tau)$ takes the following values at the indicated cusps of $\Gamma_1(7)$:
\begin{align*}
&h(\infty i) = \infty, \ \ \textstyle h(\frac{2}{7}) = 1, \ \ h(\frac{3}{7}) = 0;\\
&h(0) = \textstyle \frac{1}{1-r}, \ \ h(\frac{-1}{2}) = \frac{r-1}{r}, \ \ h(\frac{-1}{3}) = r.
\end{align*}
\label{cor:2}
\end{cor}

\begin{proof}
These values follow easily from the transformation formulas in Theorems \ref{thm:3} and \ref{thm:4}, starting with $h(\frac{2}{7}) = 1$ and $h(0) = \frac{1}{1-r}$.
\end{proof}

\begin{thm} 
We have the identity
\begin{equation}
\frac{\eta^4(\tau)}{\eta^4(7\tau)} = \frac{h^3-8h^2+5h+1}{h(h-1)},\ \ h = h(\tau), \ \textrm{Im}(\tau) > 0.
\label{eqn:11}
\end{equation}
\label{thm:5}
\end{thm}

\begin{proof}
We use the fact that $[\Gamma_0(7): \Gamma_1(7) \cup (-I)\Gamma_1(7)]=3$, from which it follows that $1, A, A^2$ are representatives for the cosets of $\Gamma_1[7] = \Gamma_1(7) \cup (-I)\Gamma_1(7)$ in $\Gamma_0(7)$.  It follows that the function
\begin{align}
\notag z(\tau) &= h(\tau)+h(A(\tau))+h(A^2(\tau))\\
& = h(\tau)+\frac{h(\tau)-1}{h(\tau)}+\frac{1}{1-h(\tau)} = \frac{h^3(\tau)-3h(\tau)+1}{h(\tau)(h(\tau)-1)}
\label{eqn:12}
\end{align}
is a modular function for $\Gamma_0(7)$.  Equation (\ref{eqn:12}) implies that the pole divisor of $z(\tau)$ in $\textsf{K}_{\Gamma_1(7)}$ is $\mathfrak{p}_\infty \mathfrak{p}_0 \mathfrak{p}_1$, so that $[\textsf{K}_{\Gamma_1(7)}:k(z(\tau))] = 3$ and $z(\tau)$ is a Hauptmodul for $\textsf{K}_{\Gamma_0(7)}$.  This, together with the $q$-expansion
$$z(\tau) = \frac{1}{q} + 4 + 2q + 8q^2 - 5q^3 - 4q^4 - 10q^5 + 12q^6 - 7q^7 + \cdots,$$
implies that the prime divisors $\mathfrak{p}_\infty, \mathfrak{p}_0, \mathfrak{p}_1$ of $\textsf{K}_{\Gamma_1(7)}$ lie above the prime divisor $\wp_\infty'$ of $\textsf{K}_{\Gamma_0(7)}$ corresponding to $\tau = \infty i$.  Hence, the prime divisors $\mathfrak{p}_r, \mathfrak{p}_{r'}, \mathfrak{p}_{r''}$ lie above the prime divisor $\wp_0'$ corresponding to the other cusp $\tau = 0$ of $\Gamma_0(7)$. (Note that $\textsf{K}_{\Gamma_1(7)}$ is a cyclic cubic extension of $\textsf{K}_{\Gamma_0(7)}$.)  Now $h(\tau) \equiv r$ modulo $\mathfrak{p}_r$ implies that
$$z(\tau) \equiv \frac{r^3-3r+1}{r(r-1)} = \frac{r^3-8r^2+5r+1}{r(r-1)} + 8 = 8 \ (\textrm{mod} \ \mathfrak{p}_r),$$
and therefore $z(\tau) \equiv 8$ (mod $\wp_0'$).  It is clear that $\left(\frac{\eta(\tau)}{\eta(7\tau)}\right)^4$ is a Hauptmodul for $\Gamma_0(7)$  whose value at the cusp $0$ is $0$ \cite[pp. 46, 51]{sch} (by the same argument as in the proof of Lemma \ref{lem:1}), and therefore $z(\tau) = \left(\frac{\eta(\tau)}{\eta(7\tau)}\right)^4+8$.  This shows that
\begin{equation*}
\left(\frac{\eta(\tau)}{\eta(7\tau)}\right)^4 =  \frac{h^3(\tau)-3h(\tau)+1}{h(\tau)(h(\tau)-1)}-8 = \frac{h^3(\tau)-8h^2(\tau)+5h(\tau)+1}{h(\tau)(h(\tau)-1)},
\end{equation*}
which is (\ref{eqn:11}).  See \cite[(4.24), p. 89]{elk}.  This proof is an expanded version of the same proof given in \cite[Appendix]{m}.
\end{proof}

\section{Proof of Ramanujan's identity (\ref{eqn:4}).}
  
We have proved the relation (\ref{eqn:3}) in Corollary \ref{cor:1}.
 We also note the relation
 \begin{equation}
 uvw = \eta(\tau) \eta^2(7\tau).
 \label{eqn:22}
 \end{equation}
 \bigskip
Hence, (\ref{eqn:4}) is equivalent to
\begin{equation}
\label{eqn:23} v^3 w^2 - w^3 u^2 + u^3 v^2 = \eta^3(\tau)\eta^2(7\tau) \cdot j_7^*(\tau)^{1/3}.
\end{equation}

We start the proof of (\ref{eqn:23}) with the following. \medskip

\begin{lem}
\begin{equation}
\label{eqn:24} j_7^*(\tau) = \frac{(h^2-h+1)^3}{h(h-1)(h^3-8h^2+5h+1)}, \ \ h = h(\tau).
\end{equation}
\label{lem:2}
\end{lem}

\begin{proof}
This follows directly from (2) and Theorem \ref{thm:5}:
\begin{align*}
j_7^*(\tau) & = \left(\frac{\eta(\tau)}{\eta(7\tau)}\right)^4+13 + 49 \left(\frac{\eta(7\tau)}{\eta(\tau)}\right)^4\\
& = \frac{h^3-8h^2+5h+1}{h(h-1)}+13+49\frac{h(h-1)}{h^3-8h^2+5h+1}\\
& = \frac{(h^3-8h^2+5h+1)^2+49h^2(h-1)^2+13h(h-1)(h^3-8h^2+5h+1)}{h(h-1)(h^3-8h^2+5h+1)}\\
& = \frac{(h^3 - 8h^2 + 5h + 1)(h^3 + 5h^2 - 8h + 1)+49h^2(h-1)^2}{h(h-1)(h^3-8h^2+5h+1)}\\
& = \frac{(h^2 - h + 1)^3}{h(h-1)(h^3-8h^2+5h+1)}.
\end{align*}
\end{proof}

Using Theorem \ref{thm:5} again, we insert
$$h^3-8h^2+5h+1 = h(h-1) \frac{\eta^4}{\eta_7^4}, \ \ \eta = \eta(\tau), \ \eta_7 = \eta(7\tau),$$
into (\ref{eqn:24}), giving
\begin{equation*}
j_7^*(\tau) = \frac{(h^2-h+1)^3}{h^2(h-1)^2} \frac{\eta_7^4}{\eta^4}.
\end{equation*}
Now we use
\begin{equation}
\label{eqn:25} h = \frac{uv^2}{w^3}, \ \ h-1 = \frac{u^3}{w^2v}
\end{equation}
the first of which comes from (\ref{eqn:5}) and the second which follows from (\ref{eqn:33a}):
$$h - 1 = h \times \frac{h-1}{h} = h \times h(A(\tau)) = \frac{uv^2}{w^3} \frac{s^2}{t^3} = \frac{uv^2}{w^3} \frac{u^2w}{v^3}.$$
This yields that
 \begin{align}
\notag j_7^*(\tau) &= \frac{(u^2 v^4-uv^2w^3+w^6)^3}{w^{18} h^2(h-1)^2} \frac{\eta_7^4}{\eta^4}\\
 \label{eqn:26} & = \frac{(u^2 v^4-uv^2w^3+w^6)^3}{u^8 v^2 w^8} \frac{\eta_7^4}{\eta^4}.
 \end{align}
 Now we have the algebraic relation
 \begin{equation*}
 v(u^2 v^4-uv^2w^3+w^6)-uw(u^3v^2-w^3u^2+v^3w^2) = (w^3-uv^2)(u^3 w - v^3 u + w^3 v),
 \end{equation*}
 where the right side is $0$ by Corollary \ref{cor:1}. Putting this into (\ref{eqn:26}) and using (\ref{eqn:22}) yields
 \begin{align}
\notag j_7^*(\tau) &= \frac{(u^3v^2-w^3u^2+v^3w^2)^3 u^3 w^3}{u^8 v^5 w^8} \frac{\eta_7^4}{\eta^4}\\
\notag &= \frac{(u^3v^2-w^3u^2+v^3w^2)^3}{u^5 v^5 w^5} \frac{\eta_7^4}{\eta^4}\\
\notag & = \frac{(u^3v^2-w^3u^2+v^3w^2)^3}{\eta^5 \eta_7^{10}} \frac{\eta_7^4}{\eta^4}\\
\label{eqn:27} & = \frac{(u^3v^2-w^3u^2+v^3w^2)^3}{\eta^9 \eta_7^6}.
 \end{align}
Taking cube roots in (\ref{eqn:27}) yields
 $$u^3v^2-w^3u^2+v^3w^2 =  \eta^3 \eta_7^2 j_7^*(\tau)^{1/3},$$
 which is (\ref{eqn:23}).  (Note that both sides of the last relation are asymptotic to $q^{3/8}$ as $\tau \rightarrow \infty i$.) This proves Ramanujan's identity (\ref{eqn:4}). \medskip
 
\section{Identities of Berndt and Zhang.}

As corollaries of this proof we have the following identities, which are simplifications of the identities in Example 2 of \cite[(4.22), (4.23)]{bz}.
 
 \begin{thm}
 We have the identities
 \begin{align}
 \label{eqn:33} \frac{v^2u}{w^3}+\frac{u^2w}{v^3}-\frac{w^2v}{u^3} &= \left(\frac{\eta(\tau)}{\eta(7\tau)}\right)^4+8,\\
  \label{eqn:34} \frac{w^3}{v^2u}+\frac{v^3}{u^2w}-\frac{u^3}{w^2v} &= -\left(\frac{\eta(\tau)}{\eta(7\tau)}\right)^4-5.
  \end{align}
  \label{thm:6}
  \end{thm}
  
  \noindent{\bf Remark.} The identities in \cite[Ex. 2, (4.22), (4.23)]{bz} follow from simply cubing these two identities, using the relations
  \begin{equation*}
  u = -iF(3/7;7\tau), \ \ v = -i F(2/7; 7\tau), \ \ w = -i F(1/7;7\tau),
  \end{equation*}
  where
  $$F(t; z) = -i \sum_{n=-\infty}^\infty{(-1)^nq^{(n+t+1/2)^2/2}}, \ \ q = e^{2\pi i z},$$
  as in \cite[Eq. (2.4), p. 227]{bz}.
  
  \begin{proof}
  From (\ref{eqn:25}) and (\ref{eqn:12}) the left side of (\ref{eqn:33}) is
  \begin{equation*}
  \frac{v^2u}{w^3}+\frac{u^2w}{v^3}-\frac{w^2v}{u^3} = h+\frac{h-1}{h}-\frac{1}{h-1} = z(\tau) = \left(\frac{\eta(\tau)}{\eta(7\tau)}\right)^4+8.
  \end{equation*}
  Similarly, the left side of (\ref{eqn:34}) is
  \begin{align*}
  \frac{w^3}{v^2u}+&\frac{v^3}{u^2w}-\frac{u^3}{w^2v} = \frac{1}{h}+\frac{h}{h-1}-(h-1) = -\frac{h^3-3h^2+1}{h(h-1)}\\
  &= 3-\frac{h^3-3h+1}{h(h-1)} = 3-z(\tau) =  -\left(\frac{\eta(\tau)}{\eta(7\tau)}\right)^4-5.
  \end{align*}
  \end{proof}
  
  This proof makes clear that the left-hand sides of (\ref{eqn:33}) and (\ref{eqn:34}) are the traces of the functions $h, 1/h \in \textsf{K}_{\Gamma_1(7)}$ to the field $\textsf{K}_{\Gamma_0(7)}$.  The fact that the individual terms on the left-hand side of (\ref{eqn:34}) -- which are reciprocals of the terms in (\ref{eqn:33}) -- produce a similar expression to the right side of (\ref{eqn:33}), is a consequence of the fact that the function $h$ is a unit in the integral closure of $k\big[\left(\frac{\eta(\tau)}{\eta(7\tau)}\right)^4\big]$ in $\textsf{K}_{\Gamma_1(7)}$, by Theorem \ref{thm:5}.  The same idea can be used to prove the following identities.
  
  \begin{thm} We have the identities
  \begin{align*}
  \frac{v^4u^2}{w^6}+\frac{u^4w^2}{v^6}+\frac{w^4v^2}{u^6} &= \left(\frac{\eta(\tau)}{\eta(7\tau)}\right)^8+14\left(\frac{\eta(\tau)}{\eta(7\tau)}\right)^4+54,\\
  \frac{w^6}{v^4u^2}+\frac{v^6}{u^4w^2}+\frac{u^6}{w^4v^2} &=  \left(\frac{\eta(\tau)}{\eta(7\tau)}\right)^8+12\left(\frac{\eta(\tau)}{\eta(7\tau)}\right)^4+41.
  \end{align*}
  \label{thm:7}
  \end{thm}
  
   \begin{proof}
  The first identity follows from (\ref{eqn:11}) and
  \begin{align*}
   \frac{v^4u^2}{w^6}+\frac{u^4w^2}{v^6}+\frac{w^4v^2}{u^6} &= h^2 + \left(\frac{h}{h-1}\right)^2+\left(\frac{1}{1-h}\right)^2\\
   & = \frac{h^6-2h^5+2h^4-4h^3+7h^2-4h+1}{h^2(h-1)^2}\\
   & = \left(\frac{h^3-8h^2+5h+1}{h(h-1)}\right)^2+14\frac{h^3-8h^2+5h+1}{h(h-1)}+54.
   \end{align*}
   The second is proved in the same way.
   \end{proof}  
    
 In a similar way, we can prove the identities from Example 3 of \cite[(4.24), (4.25)]{bz}.
 
 \begin{thm}
  We have the identities
 \begin{align}
 \label{eqn:35} \frac{u^4}{v^2}+\frac{w^4}{u^2}+\frac{v^4}{w^2} &= 2\eta(\tau)^2j_7^*(\tau)^{1/3},\\
  \label{eqn:36} \frac{v^2}{u^4}+\frac{u^2}{w^4}+\frac{w^2}{v^4} &= \frac{\eta^2(\tau)}{\eta^4(7\tau)} j_7^*(\tau)^{2/3}.
  \end{align}
  \label{thm:8}
  \end{thm}
  
  \noindent {\bf Remark.} Cubing these identities yields the identities \cite[(4.24),(4.25)]{bz}.
  
  \begin{proof}
  I will give two proofs of (\ref{eqn:35}).  First note that this identity is equivalent to
  $$ \frac{u^4}{v^2}+\frac{w^4}{u^2}+\frac{v^4}{w^2} = 2\eta(\tau) \eta(7\tau)^2 \frac{\eta(\tau)}{\eta^2(7\tau)} j_7^*(\tau)^{1/3},$$
  and by (\ref{eqn:4}) and (\ref{eqn:22}) this is equivalent to
  \begin{equation}
  \frac{u^4}{v^2}+\frac{w^4}{u^2}+\frac{v^4}{w^2} = 2uvw \left(\frac{v}{u^2}-\frac{w}{v^2}+\frac{u}{w^2}\right).
  \label{eqn:37}
  \end{equation}
  Let
  $$A = \frac{u^4}{v^2}+\frac{w^4}{u^2}+\frac{v^4}{w^2}, \ \ B = uvw \left(\frac{v}{u^2}-\frac{w}{v^2}+\frac{u}{w^2}\right).$$
  For the first proof, we simply factor the expression
  $$A-2B = \frac{(u^3w-v^3u+w^3v)^2}{u^2 v^2 w^2};$$
  this is $0$ by Corollary \ref{cor:1}.  Hence $A = 2B$.  For the second proof, we write $A$ and $B$ in terms of the function $h$.  Using $uv^2 = w^3h$ we have
  \begin{align*}
  A & = \frac{u^6w^2+u^2v^6+w^6v^2}{u^2 v^2 w^2} = \frac{(v^3 u-w^3v)^2+u^2v^6+w^6v^2}{u^2 v^2 w^2}\\
  & = \frac{2v^2}{u^2 v^2 w^2}(u^2v^4-uv^2w^3+w^6) = \frac{2}{u^2 w^2}(h^2w^6-hw^6+w^6)\\
  & = \frac{2w^4}{u^2}(h^2-h+1).
  \end{align*}
  Similarly,
  \begin{align*}
  B & = \frac{uvw}{u^2 w^2}\left(vw^2-\frac{w^3u^2}{v^2}+u^3\right)\\
  & = \frac{uvw}{u^2 w^2}\left(\frac{u^3}{h-1}-\frac{u^3}{h}+u^3\right)\\
  & = \frac{u^2v}{w}\left(\frac{1}{h-1}-\frac{1}{h}+1\right)\\
  & = \frac{u^2v}{w} \left(\frac{h^2-h+1}{h(h-1)}\right).
  \end{align*}
  Now from (\ref{eqn:25}) we have that $h(h-1) = \frac{u^4 v}{w^5}$,
 which gives that
 $$B = \frac{u^2v}{w} \frac{w^5}{u^4v} (h^2-h+1) = \frac{w^4}{u^2} (h^2-h+1).$$
  Thus, we find again that $A = 2B$.  This proves (\ref{eqn:35}). \medskip
  
  To prove (\ref{eqn:36}), note that this identity is equivalent to
  $$\frac{v^2}{u^4}+\frac{u^2}{w^4}+\frac{w^2}{v^4} = \left(\frac{\eta(\tau)}{\eta^2(7\tau)} j_7^*(\tau)^{1/3}\right)^2,$$
  which is equivalent to
  $$\frac{v^2}{u^4}+\frac{u^2}{w^4}+\frac{w^2}{v^4} = \left(\frac{v}{u^2}-\frac{w}{v^2}+\frac{u}{w^2}\right)^2.$$
  However,
  $$\frac{v^2}{u^4}+\frac{u^2}{w^4}+\frac{w^2}{v^4} - \left(\frac{v}{u^2}-\frac{w}{v^2}+\frac{u}{w^2}\right)^2 = \frac{2(u^3w-v^3u+w^3v)}{u^2 v^2 w^2} = 0.$$
  This proves (\ref{eqn:36}).
  \end{proof}

\noindent Dept. of Mathematical Sciences, LD 270 \smallskip

\noindent Indiana University - Purdue University at Indianapolis \smallskip

\noindent 402 N. Blackford St., Indianapolis, Indiana, USA, 46202. \smallskip

\noindent e-mail: pmorton@iupui.edu

\end{document}